\documentclass[11pt,leqno]{amsart}
\usepackage{amsmath}
\usepackage{amsthm}
\usepackage{amssymb}
\usepackage{graphicx}
\usepackage{amsfonts}
\usepackage{latexsym, wasysym}
\usepackage{hyperref,subcaption}
\usepackage{float}
\usepackage{makecell}

\graphicspath{{IMAGES/}}

\usepackage{enumitem}

\usepackage[font=footnotesize,labelfont=bf]{caption}

\usepackage[table]{xcolor}
\usepackage{cite}
\usepackage{multicol}
\usepackage{multirow}

\theoremstyle{plain}
\newtheorem{theorem}{Theorem}

\newtheorem{lemma}[theorem]{Lemma}

\theoremstyle{definition}

\newcommand{\K}{\mathbb{K}}

\let\oldenumerate=\enumerate
	\def\enumerate{
	\oldenumerate
	\setlength{\itemsep}{5pt}
	}
\let\olditemize=\itemize
	\def\itemize{
	\olditemize
	\setlength{\itemsep}{5pt}
	}

\allowdisplaybreaks

\usepackage[a4paper, lmargin=0.1666\paperwidth, rmargin=0.1666\paperwidth, tmargin=0.1111\paperheight, bmargin=0.1111\paperheight]{geometry} 

\usepackage[all]{nowidow} 

\usepackage{lipsum} 
\usepackage{comment}

\begin{document}

\title{Explicit Interval Estimates for Prime Numbers}

    \author[M.~Cully-Hugill]{Michaela Cully-Hugill}
	\address{School of Science, UNSW Canberra, Northcott Drive, Australia ACT 2612} 
	\email{m.cully-hugill@unsw.edu.au}

	\author[E.~S.~Lee]{Ethan S. Lee}
	\address{School of Science, UNSW Canberra, Northcott Drive, Australia ACT 2612} 
	\email{ethan.s.lee@student.adfa.edu.au}
	
\thanks{}

\subjclass[2020]{}


\begin{abstract}
Using a smoothing function and recent knowledge on the zeros of the Riemann zeta-function, we compute pairs of $(\Delta,x_0)$ such that for all $x \geq x_0$ there exists at least one prime in the interval $(x(1-\Delta^{-1}),x]$. This arXiv version has been updated with corrections to the paper \cite{CH_L_22}.
\end{abstract}

\maketitle

\section{Introduction}

In 1845, Bertrand \cite{Bertrand_45} postulated that there is at least one prime $p$ satisfying $n < p < 2n-2$ for every integer $n > 3$. In 1850, Chebyshev proved Bertrand's postulate and this kick-started an area of research into the existence of primes in intervals. We say that an interval $[x,x+h]$ is \textit{short} if $h = o(x)$ and \textit{long} if $h = \Omega(x)$, hence Bertrand's postulate is a long interval estimate.

There are a number of applications for long interval estimates. Although longer for large $x$, they can be smaller than short intervals for sufficiently small $x$, and so can verify primes in short intervals for small $x$; see \cite{Dudek_16} for an example. Long interval results have also been used in problems from additive prime number theory. For example, Dressler \cite{Dressler, DresslerAddendum} refined Bertrand's postulate to show that every positive integer $n\not\in\{1, 2, 4, 6, 9\}$ can be written as the sum of distinct odd primes. Another example is for the ternary Goldbach conjecture: before Helfgott \cite{Helfgott} offered a full proof, Ramar\'{e} and Saouter \cite{RamareSaouter} verified the ternary Goldbach conjecture up to $10^{22}$. For these and other reasons, mathematicians have sought to improve, generalise, or refine Bertrand's postulate.
For instance, Shevelev \textit{et al.} \cite{Shevelev} showed that the only integers $k\leq 10^8$ for which there exists a prime in $(k n, (k+1) n)$ for every integer $n > 1$ are $k\in\{1,2,3,5,9,14\}$; this is an extension of \cite{Bachraoui, Loo, Nagura}. 

There are several types of short interval estimates. For example, Axler \cite{Axler} proved there is always a prime in $(x,x(1 + 198.2/\log^4{x})]$ for $x > 1$. The smallest short interval estimate in the long run would be from Baker, Harman, and Pintz \cite{Baker}, with primes in $[x, x + x^{0.525}]$ for sufficiently large $x$. This, however, has not been made explicit, so it is not known how large $x$ would need to be. Better intervals are possible if assuming the Riemann hypothesis.  Cram{\'e}r showed that there are primes in $(x,x+c\sqrt{x}\log x]$ for sufficiently large $x$ and some constant $c>0$. Dudek \cite{Dudek_15} determined that we can take $c=3+\epsilon$ for sufficiently large $x$, which was improved in \cite{D_G_M_16} {to $c=1+\epsilon$}, and most recently reduced to $c=\frac{22}{25}$ by Carneiro, Milinovich, and Soundararajan \cite[Thm.~5]{C_M_S_19} for $x\geq 4$.

Bertrand's interval has been narrowed in the form $(x(1 - \Delta^{-1}), x]$, for all $x\geq x_0$ and some large constant $\Delta$. Chronologically, Schoenfeld \cite[Thm.~12]{Schoenfeld76} proved we can take $\Delta=16\,598$ for all $x > 2\,010\,760$; Ramar\'{e} and Saouter \cite{RamareSaouter} computed a table of $\Delta$ for  $x>x_0>4\cdot 10^{18}$ for a range of $x_0$, using analytic estimates and a smoothing argument; and Kadiri and Lumley \cite{KadiriLumley} refined Ramar\'{e} and Saouter's method. A selection of $\Delta$ from the latter two are given in Table \ref{comparison_table}. It is important to note that for $x\leq 4\cdot 10^{18}$ there is no need for interval estimates, as Oliveira e Silva, Herzog, and Pardi \cite{O_H_P_14} have computed the gaps between primes in this range.

In this article, we update Kadiri and Lumley's work \cite{KadiriLumley} with the most recent explicit estimates for the Riemann zeta-function $\zeta(s)$ and the Chebyshev functions
\begin{equation*}
    \theta(x) = \sum_{p\leq x}\log{p}, \qquad \psi(x) = \sum_{r= 1}^{\left[\frac{\log{x}}{\log{2}}\right]}\theta(x^{\frac{1}{r}}).
\end{equation*}
In particular, we use the zero-density estimate from Kadiri, Lumley, and Ng \cite{K_L_N_2018}, the zero-counting function estimate from Platt and Trudgian \cite{P_T_15}, the zero-free region from Mossinghoff and Trudgian \cite{MossinghoffTrudgian2015}, the Riemann height from Platt and Trudgian \cite{PlattTrudgianRH}, and bounds on $\psi(x)-\theta(x)$ from Costa Pereira \cite{Costa} and Broadbent \textit{et al.} \cite{BKLNW_20}. We also use the $L$-functions and Modular Forms Database (LMFDB) \cite{lmfdb} to extend previous computation of sums over zeros of $\zeta(s)$ and a new generalised version of Ramar\'{e} and Saouter's smooth weight. The resulting problem is one of numeric optimisation, in that we seek to maximise $\Delta$ subject to a constraint, defined over several parameters. These improvements lead to Theorem \ref{thm:main}.

\begin{theorem}\label{thm:main}
For each pair $(\Delta, x_0)$ in Table \ref{table1}, there exists at least one prime in the interval $(x(1 - \Delta^{-1}), x]$ for all $x\geq x_0$.
\end{theorem}

Table \ref{comparison_table} gives a selection of admissible $\Delta$ for Theorem \ref{thm:main}, and the most recent improvements to $\Delta$ from previous work.
\begin{table}[H]
\centering
\begin{tabular}{l|ccc}
\multirow{2}{*}{$\log x_0$} & \multicolumn{3}{c}{$\Delta$} \\ 
\cline{2-4} & \multicolumn{1}{c}{\cite[Thm.~2]{RamareSaouter}} & \multicolumn{1}{c}{\cite[Thm.~1.1]{KadiriLumley}} & \multicolumn{1}{c}{Theorem \ref{thm:main}} \\ 
\hline
46 & $8.13\cdot 10^7$ & $1.48\cdot 10^8$ & $1.63\cdot 10^8$ \\
50 & $1.90\cdot 10^8$ & $7.53\cdot 10^8$ & $1.06\cdot 10^{9}$ \\
55 & $2.07\cdot 10^8$ & $1.77\cdot 10^9$ & $1.02\cdot 10^{10}$ \\
60 & $2.09\cdot 10^8$ & $1.96\cdot 10^9$ & $7.69\cdot 10^{10}$ \\
150 & $2.12\cdot 10^8$ & $2.44\cdot 10^9$ & $2.07\cdot 10^{11}$
\end{tabular}
\caption{Values of $\Delta$ for select $x_0$ in Theorem \ref{thm:main}, compared to those in \cite{RamareSaouter} and \cite{KadiriLumley}.}
\label{comparison_table}
\end{table}

\noindent The improvement in our results largely comes from the higher Riemann height, the zero-density estimate from \cite{K_L_N_2018} and Broadbent \textit{et al.}'s estimate for $\psi(x)-\theta(x)$.

\subsection*{Structure}

Theorem \ref{thm:main} is obtained by applying the new explicit estimates and smooth weight to Theorem \ref{thm:KLMainTheorem} (in Section \ref{sec:backgound_theory}). Theorem \ref{thm:KLMainTheorem} is an analogous result to Theorem 2.7 in \cite{KadiriLumley}, and hence largely follows the proof in \cite{KadiriLumley}.

The explicit estimates we use are detailed in Section \ref{sec:Background_results}. 
The smoothing method is detailed in Section \ref{sec:backgound_theory}, with corresponding result Theorem \ref{thm:KLMainTheorem}. We implement the new explicit estimates in Section \ref{sec:smoothing_fn_etc}, and also define the new smooth weight. We also show how this weight affects the smoothing argument. Finally, we compute pairs of $(\Delta,x_0)$ in Section \ref{sec:results}, and provide some commentary on the optimisation problem. We lastly discuss avenues for future research.

\subsection*{Acknowledgements}

We extend our gratitude to Tim Trudgian for suggesting this project to us and for his advice throughout its progress. We also thank Richard Brent, David Platt, and Olivier Ramar\'{e} for their correspondence and suggestions.

\section{Prerequisite Results}\label{sec:Background_results}

Our problem is initially defined in terms of $\theta(x)$, but can be translated to $\psi(x)$ using estimates such as from Costa Pereira \cite[Thm.~5]{Costa} 
\begin{equation}\label{Costa_diff}
    \psi(x)-\theta(x) > 0.999x^\frac{1}{2} + x^\frac{1}{3}
\end{equation}
for all $x\geq e^{38}$, and Broadbent \textit{et al.} \cite[Cor.~5.1]{BKLNW_20},
\begin{equation}\label{Broadbent_diff}
    \psi(x) - \theta(x) < \alpha_1 x^\frac{1}{2} + \alpha_2 x^\frac{1}{3},
\end{equation}
with $\alpha_1= 1+ 1.93378 \cdot 10^{-8}$ and $\alpha_2 = 1.04320$, for all $x\geq e^{40}$. The latter is an improvement on the previous upper bound used in \cite{KadiriLumley}, from Costa Pereira \cite{Costa}.

In working with $\psi(x)$ we can use the truncated explicit Riemann--von Mangoldt formula \cite[\S 17]{Davenport}. One of the terms in this formula is a sum over the non-trivial zeros of $\zeta(s)$, which can be estimated using zero-density estimates and zero-free regions. We detail the most recent of these estimates in the following sections. Henceforth, unless otherwise stated, suppose that $s=\sigma+it$, $\rho=\beta+i\gamma$ is a non-trivial zero of $\zeta(s)$ with $0<\beta<1$, and $Z(\zeta)$ denotes the set of non-trivial zeros of $\zeta(s)$. 

\subsection{The Riemann height}

The Riemann Hypothesis (RH) states that $\beta=\frac{1}{2}$ for all $\rho$ in the critical strip. If the RH is known to be true for $|\gamma|\leq H$, then $H$ is called an admissible \textit{Riemann height}.
Platt and Trudgian \cite{PlattTrudgianRH} have recently announced that $H = 3,000,175,332,800$ is an admissible Riemann height, and this is the value we will use in our computations.

\subsection{Zero-free region of the Riemann zeta function}

Vall\'{e}e Poussin \cite{ValeePoussin} proved that there exists a constant $R_0>0$ such that $\zeta(s)\neq 0$ in the region
\begin{equation}\label{eqn:dlvp}
    \sigma \geq 1-\frac{1}{R_0\log t}
\end{equation}
for all $|t|\geq T$. The best $R_0$ to date is from Mossinghoff and Trudgian \cite{MossinghoffTrudgian2015}, of $R_0=5.573412$ and $T=2$.

Koborov \cite{koborov58} and Vinogradov \cite{vinogradov58} independently established an asymptotically superior zero-free region. Ford \cite{ford2002zero} made this result explicit, but it is only better than that of \cite{MossinghoffTrudgian2015} for $|t|>\exp(10\,151.5)$. Our estimates only use the zero-free region at the Riemann height $H$, hence we only use the result of \cite{MossinghoffTrudgian2015}.

\subsection{Zero-density estimates of the Riemann zeta-function}

Zero-density estimates bound the number of zeros of $\zeta(s)$ in some rectangular area of the critical strip. The typical notation is, for $\frac{1}{2}\leq \sigma<1$,
\begin{align*}
    N(\sigma, T) &= \#\{\rho:\sigma < \beta < 1, 0 < \gamma < T\}.
\end{align*}
The number of zeros in the full critical strip is denoted $N(T)$.
Backlund \cite{Backlund_18} established that $N(T)=P(T) + O(\log T)$, where
\begin{equation*}
    P(T) := \frac{T}{2\pi}\log{\frac{T}{2\pi}} - \frac{T}{2\pi} + \frac{7}{8}.
\end{equation*}
This has been made explicit, in that there are constants $a_1, a_2, a_3$ such that
\begin{equation}\label{classical_T}
    |N(T)-P(T)|\leq R(T)
\end{equation}
where $R(T) = a_1\log{T} + a_2\log\log{T} + a_3$ for all $T \geq T_0$. A summary of previous estimates for $R(T)$ is given in \cite{Trudgian_14}, and Trudgian proves that we can take
$$(a_1, a_2, a_3) = (0.112, 0.278, 2.510)$$
for $T_0=e$. Platt and Trudgian \cite[Cor.~1]{P_T_15} improved these values to\footnote{{One could potentially improve our computations using Hasanalizade \textit{et al.}'s results in \cite{HasanalizadeShenWongRiemann}, however we expect any changes to be marginal. At the time of writing this paper, \cite{HasanalizadeShenWongRiemann} had not been released, so we proceeded with Platt and Trudgian's result from \cite{P_T_15}.}}
$$(a_1, a_2, a_3) = (0.110, 0.290, 2.290).$$

For $N(\sigma, T)$, the best explicit estimate for $\sigma$ near 1 is from Kadiri, Lumley, and Ng \cite{K_L_N_2018}. Table 1 of \cite{K_L_N_2018} give values of $A(\sigma)$ and $B(\sigma)$ such that for $T\geq H$, $k\in \left[ 10^9 / H, 1\right]$, and $\sigma>1/2$, we have
\begin{equation}\label{zd_K}
    N(\sigma, T) \leq A(\sigma) \left( \log(kT) \right)^{2\sigma} (\log T)^{5-4\sigma}T^{\frac{8}{3}(1-\sigma)} + B(\sigma) (\log T)^2 := D(\sigma,T).
\end{equation}
To give an example, for $\sigma = 0.9$ we can take $A(\sigma) = 11.499$ and $B(\sigma) = 3.186$. It is possible to recalculate $A$ and $B$ with Platt and Trudgian's Riemann height $H$, using the expressions in (4.72) and (4.73) of \cite{K_L_N_2018}. This effects a drop in both constants, so we use these constants in \eqref{zd_K} for further calculations.

Simoni\v{c} \cite{Simonic_19} has given an asymptotically smaller estimate than (\ref{zd_K}) for $\sigma$ near $1/2$. We have not used this estimate because the zero-density estimates are only used at the Riemann height, at which point (\ref{classical_T}) is smaller than that of \cite{Simonic_19} for $\sigma\in [1/2,5/8]$ --- after which (\ref{zd_K}) is better than both.

\section{The Main Theorem}\label{sec:backgound_theory}

In what follows, let $y=(1-\Delta^{-1})x$. To prove that $(\Delta, x_0)$ is a pair for which there exists a prime in $(y, x]$ for all $x\geq x_0$, we will ensure
\begin{equation}\label{condition_theta}
    \theta(x) - \theta(y) > 0
\end{equation}
for $x\geq x_0$. Kadiri and Lumley use Ramar\'{e} and Saouter's smoothing argument from \cite{RamareSaouter}, which involves splitting the interval into two parts: one to be estimated with the Riemann--von Mangoldt explicit formula, and the other with the Brun--Titchmarsh theorem. The latter is useful for minimising the number of primes near the endpoints of the interval. We will also follow this argument.

\subsection{Definitions}

Throughout the remainder of the paper, suppose that
\begin{itemize}
    \item $m\geq 2$ is an integer,
    \item $u=\delta/m$ for $0\leq \delta \leq 10^{-6}$,
    \item $0 \leq a < 1/2$, and
    \item $X\geq X_0 \geq 3.99\cdot 10^{18}$.
\end{itemize}
Note that the bound on $X_0$ is to allow $x_0$ to be as small as $4\cdot 10^{18}$, given the bounds on $m$, $\delta$, and $a$.
We also set $x=e^u X(1+\delta(1-a))$, $y = X(1+\delta a)$, and $x_0= e^u X_0(1+\delta(1-a))$; it follows that
\begin{equation*}
    \Delta = \left(1-\frac{1+\delta a}{e^u (1+\delta(1-a))}\right)^{-1}.
\end{equation*}

As in \cite{RamareSaouter}, suppose $f$ is a smooth weight that is `$m$-admissible' over $[0, 1]$. We require this so that:
\begin{itemize}
    \item $f$ is $m$-times differentiable,
    \item $f^{(k)}(0) = f^{(k)}(1) = 0$ for $0 \leq k \leq m - 1$,
    \item $f$ is non-negative, and
    \item $f$ is not identically zero.
\end{itemize}
Further to this condition, we use the notations
\begin{align*}
   ||f||_1 &= \int_0^1 |f(t)| dt, 
   \qquad ||f||_2 = \left( \int_0^1 |f(t)|^2 dt\right)^\frac{1}{2}, \\
   \nu(f,a) &= \int_0^a f(t) dt + \int_{1-a}^1 f(t) dt,\\
   \Xi_{F}(u, X, \delta, t) &= F(\exp(u) X(1+\delta t)) - F(X(1+\delta t)), \\
   I(\delta, u, X) &= \frac{1}{||f||_1} \int_0^1 \Xi_{\theta}(u, X, \delta, t) f(t) dt,\\
   J(\delta, u, X) &= \frac{1}{||f||_1} \int_0^1 \Xi_{\psi}(u, X, \delta, t) f(t) dt.
\end{align*}

\subsection{Statement of the main theorem}

For all $t\in [a, 1-a]$ we have
\begin{equation*}
    \Xi_{\theta}(u, X, \delta, t) \leq \theta(x) - \theta(y).
\end{equation*}
Multiplying both sides by $f(t)$, and integrating over $t\in [a, 1-a]$ gives
$$\theta(x) - \theta(y)\geq I(\delta, u, X) - \frac{1}{||f||_1} \left( \int_0^a + \int_{1-a}^1 \right) \Xi_{\theta}(u, X, \delta, t) f(t) dt.$$  

We can estimate $\Xi_{\theta}$ with Montgomery and Vaughan's version of the Brun--Titchmarsh theorem in \cite[Thm.~2]{M_V_73}, which states that for $M>0$ and $N>1$, 
\begin{equation}\label{BT}
\pi(M+N)-\pi(M) \leq \frac{2N}{\log N}.
\end{equation}
Then, for $X>(e^u-1)^{-1}$, we have
\begin{align*}
    \Xi_{\theta}(u, X, \delta, t) \leq \frac{2(e^u-1)(1 + \delta t) X \log(e^u(1+\delta t) X)}{\log((e^u - 1)(1+\delta t)X)};
\end{align*}
note that this expression increases in $t$. It follows that
\begin{align*}
    \left( \int_0^a + \int_{1-a}^1 \right) \Xi_{\theta}(u, X, \delta, t) f(t) dt &\leq E(X),
\end{align*}
where 
\begin{align}\label{E(X)}
    E(X) = \frac{2(e^u-1)(1 + \delta) X\log(e^u(1+\delta) X) \nu(f,a)}{\log((e^u - 1)(1+\delta)X)}.
\end{align}
Using \eqref{Costa_diff}, \eqref{Broadbent_diff}, and the bounds on $m$, $\delta$, $a$, $X_0$, we arrive at
\begin{equation*}
    I(\delta, u, X) \geq J(\delta, u, X) - \omega\sqrt{X},
\end{equation*}
with $\omega = 1.0344\cdot 10^{-3}$. The explicit Riemann--von Mangoldt formula can be used to estimate $J(\delta, u, X)$. This process would be identical to the one in \cite{KadiriLumley}, which involves splitting a sum over the non-trivial zeros of $\zeta(s)$, and estimating these with zero-density estimates and a zero-free region. Therefore, we can directly take the bound for $J(\delta, u, X)$ from \cite{KadiriLumley} given by equations (18), (19), and Lemma 2.6. With this, we arrive at Theorem \ref{thm:KLMainTheorem}.

\begin{theorem}\label{thm:KLMainTheorem}
Suppose that $\gamma$ are ordinates such that $\rho = \beta + i\gamma\in Z(\zeta)$.
There exists a prime in $((1-\Delta^{-1})x,x]$ for all $X\geq X_0$ if
\begin{equation}\label{condition}
\begin{aligned}
    F(0,m,\delta)
    - B_0(m,\delta){X_0}^{-\frac{1}{2}}
    - B_1(m, \delta, T_1){X_0}^{-\frac{1}{2}}
    - B_2(m, \delta, T_1){X_0}^{-\frac{1}{2}} & \,\\
    - B_3(m, \delta, \sigma_0){X_0}^{\sigma_0-1}
    - B_3(m, \delta, 1-\sigma_0){X_0}^{-\sigma_0}
    - B_{41}(X_0, m, \delta, \sigma_0) & \, \\
    - B_{42}(m, \delta, \sigma_0){X_0}^{-1+\frac{1}{R_0\log{H}}}
    - \frac{u}{2(\exp(u) - 1){X_0}^{2}}  & \, \\
    - \frac{\omega}{(\exp(u) - 1){X_0}^{\frac{1}{2}}}
    - \frac{E(X_0)}{||f||_1 (\exp(u)-1)X_0} & > 0,
\end{aligned}
\end{equation}
where $E(X)$ is defined in (\ref{E(X)}).
The functions $B_i$ are defined as follows. For $0\leq k \leq m$, $s = \sigma + i\tau$ with $\tau > 0$, and $0\leq\sigma\leq 1$, let
\begin{equation*}
    F(k,m,\delta) = \frac{1}{||f||_1} \int_0^1 (1 + \delta t)^{1+k} |f^{(k)}(t)| dt.
\end{equation*}
First,
\begin{align*}
    B_0(m,\delta) = \min\left\{\frac{4 F(0,m,\delta)}{\exp(u/2) + 1} N_0, \frac{4 F(1,m,\delta)}{(\exp(u/2) + 1)\delta} S_0\right\},
\end{align*}
in which $N_0$ denotes the number of zeros $\rho\in Z(\zeta)$ such that $0<\gamma \leq T_0$ and
\begin{equation*}
    \sum_{0<\gamma \leq T_0}\frac{1}{\gamma} \leq S_0.
\end{equation*}
Second,
\begin{align*}
    B_1(m,\delta,T_1) = \min\left\{\frac{4 F(0,m,\delta)}{\exp(u/2) + 1} \left(N(T_1) - N_0\right), \frac{4 F(1,m,\delta)}{(\exp(u/2) + 1)\delta} S_1(T_0, T_1)\right\},
\end{align*}
in which
\begin{equation*}
    \sum_{T_0 < \gamma \leq T_1} \frac{1}{\gamma} \leq S_1(T_0, T_1).
\end{equation*}
Third,
\begin{equation*}
    B_2(m,\delta,T_1) = \frac{2 F(m,m,\delta)}{(\exp(u/2) - 1)\delta^{m}} S_2(m,T_1),\text{ in which } \sum_{T_1 < \gamma \leq H} \frac{1}{{\gamma}^{m+1}} \leq S_2(m, T_1).
\end{equation*}
Fourth,
\begin{equation*}
    B_3(m,\delta,\sigma) = \frac{2 F(m,m,\delta) (\exp(u\sigma) + 1)}{(\exp(u) - 1)\delta^{m}} S_3(m),\text{ in which } \sum_{\gamma > H} \frac{1}{{\gamma}^{m+1}} \leq S_3(m).
\end{equation*}
Fifth,
\begin{align*}
    B_{41}(X_0,m,\delta,\sigma_0) &= \frac{2 F(m,m,\delta) (\exp(u) + 1)}{(\exp(u) - 1)\delta^{m}} S_5(X_0,m,\sigma_0),\\
    B_{42}(m,\delta,\sigma_0) &= \frac{2 F(m,m,\delta) (\exp(u) + 1)}{(\exp(u) - 1)\delta^{m}} S_4(m,\sigma_0),
\end{align*}
in which
\begin{equation*}
    \sum_{\substack{\sigma_0 < \beta < 1\\\gamma > H}} \frac{1}{{\gamma}^{m+1}} \leq S_4(m, \sigma_0),\qquad
    \sum_{\substack{\sigma_0 < \beta < 1\\\gamma > H}} \frac{{X_0}^{-\frac{1}{R_0\log{\gamma}}}}{{\gamma}^{m+1}} \leq S_5(X_0, m, \sigma_0).
\end{equation*}
\end{theorem}
\section{The Smoothing Function and Important Estimates}\label{sec:smoothing_fn_etc}

To compute the $B_i$ functions in \eqref{condition}, we need estimates for each $S_i$ and $F(k,m,\delta)$ for $k=0,1,m$. New estimates for the $S_i$ are given in Section \ref{ssec:Si} (Lemma \ref{lem:estimates}), and bounds for $F(k,m,\delta)$ are given in Lemma \ref{F_bounds}. For the latter, we first need to choose the smooth weight, which is discussed in \ref{ssec:smoothing_fn}. In Section \ref{ssec:F_k_m_delta} we give bounds for $F(k,m,\delta)$ in Lemma \ref{F_bounds}.

\subsection{Estimating each $S_i$}\label{ssec:Si}

For $S_1$, $S_2$, $S_3$ we will use the following lemma from Brent, Platt, and Trudgian \cite[Lem.~3]{BrentPlattTrudgian}, which is a refinement of Lehman's work in \cite[Lem.~1]{Lehman}. For $S_4$ and $S_5$, we will use the zero-density estimate \eqref{zd_K} from \cite{K_L_N_2018}.

\begin{lemma}[Brent--Platt--Trudgian]\label{lem:BPT}
Suppose that $A_0 = 2.067$, $A_1 = 0.059$, $A_2 = 1/150$, $2\pi \leq U \leq V$ and $\phi : [U, V] \to [0, \infty )$ is differentiable, monotone, and non-increasing on $[U,V]$ such that $\phi'(t) \leq 0$ and $\phi''(t)\geq 0$. Then
\begin{equation*}
    {\sum}'_{\substack{\beta + i\gamma\in Z(\zeta)\\U \leq \gamma \leq V}} \phi(\gamma)
    = \frac{1}{2\pi} \int_{U}^{V} \phi(t)\log{\frac{t}{2\pi}}dt + \phi(V)Q(V) - \phi(U)Q(U) + E_2(U,V),
\end{equation*}
in which $Q(T) = N(T) - P(T)$ (as defined in \eqref{classical_T}), $\sum'$ means that if $\beta + iV \in Z(\zeta)$, then the contribution $\phi(V)$ is weighted by $1/2$, and
$$|E_2(U,V)| \leq 2(A_0 + A_1\,\log{U})|\phi'(U)| + (A_1 + A_2) \frac{\phi(U)}{U}.$$
\end{lemma} It is worth noting that the constants $A_0$ and $A_1$ are taken from Trudgian \cite[Thm.~2.2]{Trudgian_11} and $A_2$ is from Lemma 2 of \cite{BrentPlattTrudgian}. Also, to use this estimate for the standard $\sum$ instead of $\sum'$ notation, we will include the potential contribution of $\phi(V)/2$. Hence we will use $$|E_2(U,V)| \leq 2(A_0 + A_1\,\log{U})|\phi'(U)| + (A_1 + A_2) \frac{\phi(U)}{U} + \frac{\phi(V)}{2}.$$

\begin{lemma}\label{lem:estimates}
For integers $m\geq 2$, $X_0 \geq 3.99\cdot 10^{18}$, $T_1\in (T_0, H)$, and $\sigma\in \left(\frac{1}{2},1\right)$, we have
\begin{align}
    S_1(T_0,T_1) &= \frac{1}{2\pi} \log\frac{T_1}{T_0}\log\frac{\sqrt{T_0 T_1}}{2\pi} + \frac{R(T_0)}{T_0} + \frac{R(T_1)+\frac{1}{2}}{T_1} + \mathbf{E}_{T_0},\tag{S1}\label{eqn:estimateS1}\\
    S_2(m,T_1) &= \frac{1 + m\log\frac{T_1}{2\pi}}{2\pi m^2 {T_1}^m} - \frac{1 + m\log\frac{H}{2\pi}}{2\pi m^2 {H}^m} + \frac{R(T_1)}{{T_1}^{m+1}} + \frac{R(H)+\frac{1}{2}}{{H}^{m+1}} + \dot{\mathbf{E}}_{m,T_1},\tag{S2}\label{eqn:estimateS2}\\
    S_3(m) &= \frac{1 + m\log\frac{H}{2\pi}}{2\pi m^2 {H}^m} + \frac{R(H)}{{H}^{m+1}} + \ddot{\mathbf{E}}_{m,H},\tag{S3}\label{eqn:estimateS3}
\end{align}
in which
\begin{align*}
    \mathbf{E}_{T_0} &= (A_0 + A_1\log{T_0})\frac{2}{{T_0}^2} + (A_1 + A_2) \frac{1}{{T_0}^2},\\
    \dot{\mathbf{E}}_{m,T_1} &= (A_0 + A_1\log{T_1})\frac{2(m+1)}{{T_1}^{m+2}} + (A_1 + A_2) \frac{1}{{T_1}^{m+2}},\\
    \ddot{\mathbf{E}}_{m,H} &= (A_0 + A_1\log{H})\frac{2(m+1)}{{H}^{m+2}} + (A_1 + A_2) \frac{1}{{H}^{m+2}}.
\end{align*}
Moreover, we have
\begin{align}
    S_4(m,\sigma) &= \frac{D(\sigma, H)}{H^{m+1}} + \int_{H}^\infty \frac{\partial D(\sigma, t)}{\partial t} \frac{1}{t^{m+1}}\,dt,\tag{S4}\label{eqn:estimateS4}\\
    S_5(X_0, m, \sigma) &= \frac{ D(\sigma, H)}{H^{m+1}} {X_0}^{-\frac{1}{R_0\log{H}}} + \int_{H}^\infty \frac{\partial D(\sigma, t)}{\partial t}\frac{1}{t^{m+1}} \,dt . \tag{S5}\label{eqn:estimateS5}
\end{align}
\end{lemma}

\begin{proof}
Using $|Q(T)|\leq R(T)$ from \eqref{classical_T}, we have
$$|\phi(V)Q(V) - \phi(U)Q(U)| \leq \phi(V)R(V) + \phi(U)R(U)$$
by the triangle inequality.
Using Lemma \ref{lem:BPT} with $\phi(\gamma) = \gamma^{-1}$ we obtain \eqref{eqn:estimateS1}, and with $\phi(\gamma) = \gamma^{-(m+1)}$ we retrieve \eqref{eqn:estimateS2} and \eqref{eqn:estimateS3}.

Suppose that $\phi(t) = o(1)$ as $t\rightarrow \infty$ and recall that $\phi'(t) \leq 0$. Then $\phi(t) N(\sigma, t) \to 0$ as $t\to \infty$, so
\begin{equation*}
    \sum_{\substack{\gamma > H\\\sigma \leq \beta < 1}} \phi(\gamma)
    = - \int_{H}^\infty N(\sigma, t)\phi'(t)dt.
\end{equation*}
Using \eqref{zd_K} for $T\geq H$, we observe that
\begin{align*}
    - \int_{H}^\infty N(\sigma, t)\phi'(t)dt
    \leq - \int_{H}^\infty  D(\sigma, t) \phi'(t) dt
\end{align*}
and with integration by parts, we have
\begin{align*}
   - \int_{H}^\infty  D(\sigma, t) \phi'(t)dt 
    \leq  D(\sigma, H)\phi(H) + \int_{H}^\infty \frac{\partial D(\sigma, t)}{\partial t} \phi(t) dt.
\end{align*}
Combining these observations, we have
\begin{equation}\label{eqn:abcdefghi}
    \sum_{\substack{\gamma > H\\\sigma \leq \beta < 1}} \phi(\gamma)
    \leq  D(\sigma, H)\phi(H) + \int_{H}^\infty \frac{\partial D(\sigma, t)}{\partial t}\phi(t)dt.
\end{equation}
Under the choice $\phi(\gamma) = \gamma^{-(m+1)}$, we retrieve \eqref{eqn:estimateS4}. Noting that $X_0^{\frac{-1}{R_0 \log t}}\leq 1$ for $\gamma\geq 1$, choosing $\phi(\gamma) = X_0^{\frac{-1}{R_0 \log \gamma}} \gamma^{-(m+1)}$ yields \eqref{eqn:estimateS5}.
\end{proof}

\subsection{The smoothing function}\label{ssec:smoothing_fn}

The weight used in \cite{RamareSaouter} and \cite{KadiriLumley} was $f_1(t)=(4t(1-t))^m$ over $t\in[0,1]$, which is a bell curve with $\max_{[0,1]} f_1(t) = 1$. Larger $m$ narrows the curve, and shrinks the function near its endpoints. We find that this function can be generalised to $f_2(t)=(At^n(1-t))^m$, where $n$ is a positive integer, and $A=(n+1)^{n+1} n^{-n}$ to keep $\max_{[0,1]} f_2(t) = 1$. For $n\geq 2$, $f_2$ is skewed to the left, with larger $n$ amplifying this skew. We will use $f_2$ to estimate $F(k,m,\delta)$, to optimise over $n$.

To compute the constraint in Theorem \ref{thm:KLMainTheorem} we can use closed-form expressions for $F(0,m,\delta)$, $\nu(f,a)$, $||f||_1$, and $||f^{(m)}||_2$ with $f=f_2$. Integration by parts yields 
\begin{align*}
    ||f||_1 =  \int_0^1 (At^n(1-t))^m dt = \frac{A^m m! (mn)!}{(mn+m+1)!},
\end{align*}
and similarly
\begin{align*}
    \int_0^1 (1+\delta t) f(t)dt
    &= A^m \int_0^1 (t^{mn} +\delta t^{mn+1}) (1-t)^m dt \\
    &= A^m m! \int_0^1 \left( \frac{(mn)! t^{mn+m}}{(mn+m)!} +\frac{(mn+1)! \delta t^{mn+m+1}}{(mn+m+1)!} \right) dt \\
    &= A^m m! \left( \frac{(mn)!}{(mn+m+1)!} +\frac{(mn+1)! \delta}{(mn+m+2)!} \right).
\end{align*}
So,
\begin{align*}
    F(0,m,\delta) &= \frac{\int_0^1 (1+\delta t) f(t)dt}{||f||_1}\\
    &= m! \left( \frac{(mn)!}{(mn+m+1)!} +\frac{(mn+1)! \delta}{(mn+m+2)!} \right) \frac{(mn+m+1)!}{(m!)(mn)!} \\
    &= 1 +\frac{(mn+1) \delta}{mn+m+2}.
\end{align*}

For $\nu(f,a)$, integration by parts can be used on the following two integrals:
\begin{align*}
    \int_{1-a}^1 f(t) dt &= -A^m \sum_{k=0}^{m-1} \frac{m!(mn)!}{(m-k)!(mn+k+1)!}a^{m-k}(1-a)^{mn+k+1} \\
    &\qquad + \frac{A^m m!(mn)!}{(mn+m+1)!}\left(1-(1-a)^{mn+m+1}\right)
\end{align*}
and
\begin{align*}
    \int_{0}^a f(t) dt = A^m \sum_{k=0}^{m} \frac{m!(mn)!}{(m-k)!(mn+k+1)!}(1-a)^{m-k} a^{mn+k+1}.
\end{align*}
Therefore, we have
\begin{align*}
    \nu(&a,T) = A^m \sum_{k=0}^{m-1} \frac{m!(mn)!}{(m-k)!(mn+k+1)!}\left((1-a)^{m-k} a^{mn+k+1} \right.\\
    &\left. - a^{m-k}(1-a)^{mn+k+1} \right)
    + \frac{A^m m!(mn)!}{(mn+m+1)!}\left(1-(1-a)^{mn+m+1} + a^{mn+m+1} \right).
\end{align*}

Finally, to bound $||f^{(m)}||_2$, note that the binomial theorem implies
\begin{equation*}
    f_2(t) = A^m t^{mn} (1-t)^m = A^m \sum_{k=0}^m \binom{m}{k} (-1)^k t^{mn+k},
\end{equation*}
and 
\begin{equation*}
    f_2^{(2m)}(t) = A^m \sum_{k=0}^m \binom{m}{k} (-1)^k \frac{(mn+k)!}{(mn+k-2m)!} t^{mn+k-2m}.
\end{equation*}
Since $f^{(k)}(0) = f^{(k)}(1) = 0$ for $0 \leq k \leq m - 1$, integration by parts gives
\begin{align*}
    \int_0^1 \left( f^{(m)} (t)\right) ^2 dt &= \left[ f^{(m)} (t) f^{(m-1)} (t) \right]_0^1 - \int_0^1 f^{(m+1)} (t) f^{(m-1)} (t) dt \\
    &= \ldots = (-1)^m \int_0^1 f^{(2m)} (t) f(t) dt \\
    &= (-1)^m A^{2m} \sum_{k=0}^m (-1)^k \binom{m}{k} \frac{(mn+k)!m!(2mn+k-2m)!}{(mn+k-2m)!(2mn+k-m+1)!}.
\end{align*}
It follows that
\begin{equation*}
||f^{(m)}||_2 = A^m \left( m! \sum_{k=0}^m (-1)^{m+k} \binom{m}{k} \frac{(mn+k)!(2mn+k-2m)!}{(mn+k-2m)!(2mn+k-m+1)!} \right)^{1/2}.
\end{equation*}

\subsection{Estimating $F(k,m,\delta)$}\label{ssec:F_k_m_delta}

With $f=f_2$ we can bound $F(k,m,\delta)$ using the following update to \cite[Lem.~3.1]{KadiriLumley}.
\begin{lemma}\label{F_bounds}
For $B=\frac{n}{n+1}$, with integer $n\geq 1$, let
\begin{align*}
    \lambda_0(m,n,\delta) &= \frac{2(B^n - B^{n+1})^m (mn+m+1)!}{m! (mn)!}, \\
    \lambda_1(m,n,\delta) &= (1+\delta)^2 \frac{2 (B^n - B^{n+1})^m (mn+m+1)!}{m! (mn)!}, \\
    \lambda(m,n,\delta) &= \sqrt{\frac{(1+\delta)^{2m+3}-1}{\delta(2m+3)}} \frac{(mn+m+1)!}{A^m m! (mn)!} ||f^{(m)}||_2.
\end{align*}
Then, we have $\lambda_0(m,n,\delta) \leq F(1,m,\delta) \leq \lambda_1(m,n,\delta)$ and $F(m,m,\delta) \leq \lambda(m,n,\delta)$.
\end{lemma}

\begin{proof}
For $F(1,m,\delta)$, we have $$\frac{||f'||_1}{||f||_1} \leq F(1,m,\delta) \leq (1+\delta)^2 \frac{||f'||_1}{||f||_1},$$ in which \begin{align*}
||f'||_1 = \int_0^1 |f'(t)| dt &= A^m m \int_0^1 (t^n-t^{n+1})^{m-1} \left| (nt^{n-1} - (n+1) t^n) \right| dt \\
&= A^m m \left( \int_0^B + \int_B^1 \right) (t^n-t^{n+1})^{m-1} \left| (nt^{n-1} - (n+1) t^n) \right| dt, 
\end{align*}
where $B$ is chosen such that the derivative of $g(t)=t^n-t^{n+1}$ (i.e. $nt^{n-1} - (n+1) t^n$) is non-negative for $0\leq t\leq B$ and non-positive for $B\leq t\leq 1$. So, we take $B=\frac{n}{n+1}$, and find
\begin{align*}
||f'||_1 &= A^m m \left( \int_0^B g(t)^{m-1} g'(t) dt - \int_B^1 g(t)^{m-1} g'(t) dt \right).
\end{align*}
For the first integral,
\begin{align*}
\int_0^B g(t)^{m-1} g'(t) dt &= g(B)^m - \int_0^B (m-1) g(t)^{m-1} g'(t) dt \\
&= \frac{1}{m} g(B)^m.
\end{align*}
Similar logic can be applied to the second integral, so that we have
$$||f'||_1 = 2A^m (B^n - B^{n+1})^m.$$

Lastly, for $F(m,m,\delta)$, the Cauchy--Schwarz inequality implies
\begin{align*}
F(m,m,\delta) &\leq \sqrt{\int_0^1 (1+\delta)^{2(m+1)} dt}  \frac{||f^{(m)}||_2}{||f||_1} \\
&= \sqrt{\frac{(1+\delta)^{2m+3}-1}{\delta(2m+3)}} \frac{(mn+m+1)!}{A^m m! (mn)!} ||f^{(m)}||_2.\qedhere
\end{align*}
\end{proof}

\section{Results and Future Research}\label{sec:results}

\subsection{Results}

The value of $N_0$ and $S_0$ can be computed using the list of zeros from the LMFDB database, after fixing $T_0$. Taking $T_0 = 104\,537\,615$, we have $N_0 = 2.6\cdot 10^{8}$ and $S_0 = 21.98308$. It would be possible to take $T_0$ higher, but this would restrict the range for $T_1$, which appears to have an optimal value based on $m$ and $\delta$. In particular, this optimum appears to be smaller for small $x$. Hence, the choice of $T_0$ was based on what the likely range of $m$ and $\delta$ implied the optimal $T_1$ to be.

With the estimates in Section \ref{sec:smoothing_fn_etc}, we used Theorem \ref{thm:KLMainTheorem} to calculate admissible values for $\Delta$ for all $x\geq x_0$. That is, after fixing $x_0$, we sought the largest $\Delta$ from sets of $m$, $n$, $a$, $\delta$, $T_1$, and $\sigma_0$ that satisfied (\ref{condition}). Numerical optimisation was used for all parameters except $n$ and $\sigma_0$. We found that $n=1$ was in fact the optimal choice. For $\sigma_0$, the optimal value is independent of the other parameters, as it is the point at which the estimate in (\ref{zd_K}) is smaller than (\ref{classical_T}) for all $\sigma\geq \sigma_0$ at $T=H$. Using (\ref{zd_K}) at $\sigma=0.78$, with $A=5.8773$ and $B=3.869$, we found this to be the case at $\sigma_0=0.7804$.

The difficulty in optimising over $m$, $a$, $\delta$, and $T_1$, is that the optimal value of each is a function of the others. Fixing all but one parameter would define the optimal value of that parameter, and potentially conceal the `true' optimum. We used the \texttt{differential\textunderscore evolution} function in Python's \texttt{scipy.optimize} package to maximise $\Delta$ with respect to $m$, $a$, $\delta$, $T_1$, and a small range of $n$.\footnote{Our code is available \href{https://github.com/EthanSLee/Interval-Estimates-for-Prime-Numbers/tree/main}{\texttt{here}}.} This function implements a differential evolution algorithm from Storn and Price \cite{S_P_97} for finding the global extremum of a multivariate function. We chose this method because it tests a wide range of parameter value combinations before converging on a solution. This ideally addresses the problem of having many local maximums. However, like those before us, we cannot claim to have found the true optimal values for each parameter, but we have sufficient reason to believe they are close.

Table \ref{table1} lists admissible values of $\Delta$ for each $x_0$, and the corresponding parameter values. Note that each $\Delta$ is calculated using parameter values rounded to seven decimal places. These values for $\Delta$ have the largest order we could find with the differential evolution method. {The calculations were run with precision of up to 400 decimal places. We also used functions from Python's \texttt{mpmath} library, such as \texttt{fdiv}, to maintain precision when combining the smallest values, e.g. from expressions containing $e^{-\delta}$ and $\delta^{m}$.} Furthermore, some expressions in Section \ref{sec:smoothing_fn_etc} were re-arranged or approximated in order to minimise the error in calculations. It should also be noted that the values of $m$ were restricted to integers because the bounds on $F(k,m,\delta)$ required integrating some $m$ times.

\begin{table}[H]
\begin{tabular}{|l|l|c|c|c|c|c|}
\hline
$\log x_0$ & $m$ & $\delta$ & $a$ & $T_1$ & $\Delta$ \\ \hline
$\log(4\cdot 10^{18})$ & 5 & $3.341898\cdot 10^{-8}$ & $0.2173221$ & $3.388300\cdot10^{8}$ & $3.90970\cdot 10^{7}$ \\ \hline
43  & 5 & $3.123609\cdot10^{-8}$ & $0.2172087$ & $3.565573\cdot10^{8}$ & $4.18168 \cdot10^{7}$ \\ \hline
46  & 4 & $6.874386\cdot10^{-9}$ & $0.1813398$ & $1.215214\cdot 10^9$ & $ 1.63940\cdot 10^{8}$ \\ \hline
50  & 5 & $1.208702\cdot 10^{-9}$ & $0.2101882$ & $8.262901\cdot 10^9$ & $1.06120\cdot 10^{9}$ \\ \hline
55  & 9 & $1.757070\cdot 10^{-10}$ & $0.2789679$ & $9.330703\cdot 10^{10}$ & $1.02884\cdot 10^{10}$ \\ \hline
60  & 30 & $4.873014\cdot 10^{-11}$ & $0.3832708$ & $9.890872\cdot 10^{11}$ & $7.69184\cdot 10^{10}$ \\ \hline
75  & 82 & $6.286379\cdot 10^{-11}$ & $0.4603978$ & $2.844455\cdot 10^{12}$ & $1.74043\cdot 10^{11}$ \\ \hline
90  & 82 & $6.270787\cdot 10^{-11}$ & $0.4628348$ & $2.361523\cdot 10^{12}$ & $1.84304\cdot 10^{11}$ \\ \hline
105 & 80 & $6.183604\cdot 10^{-11}$ & $0.4641109$ & $ 1.860117\cdot 10^{12}$ & $1.91886\cdot 10^{11}$ \\ \hline
120 & 84 & $6.335103\cdot 10^{-11}$ & $0.4660744$ & $2.020015\cdot 10^{12}$ & $1.97917\cdot 10^{11}$ \\ \hline
135 & 90 & $6.590771\cdot 10^{-11}$ & $0.4681018$ & $2.198840\cdot 10^{12}$ & $2.02553\cdot 10^{11}$ \\ \hline
150 & 79 & $6.090246\cdot 10^{-11}$ & $0.4666782$ & $2.703636\cdot10^{12}$ & $2.07053\cdot 10^{11}$ \\ \hline
300 & 79 & $5.965946\cdot 10^{-11}$ & $0.4699104$ & $2.695460\cdot10^{12}$ & $2.30126\cdot 10^{11}$ \\ \hline
600 & 72 & $5.361814\cdot 10^{-11}$ & $0.4699322$ & $1.002066\cdot 10^{12}$ & $2.51949\cdot 10^{11}$ \\ \hline
\end{tabular}
\caption{Admissible pairs of $x_0$ and $\Delta$ in Theorem \ref{thm:main} and the corresponding parameter values.}
\label{table1}
\end{table}

The parameter values in Table \ref{table1} can be better understood by considering the relative influence of each parameter. Large $\Delta$ comes from large $m$, small $\delta$, and large $a$, but \eqref{condition} requires small $m$, large $\delta$, and small $a$ --- this is why there are optimal values for each. Of these three, $\delta$ has most impact on $\Delta$, in that it determines the order. For the constraint, $\delta$ appears to be most influential.

\subsection{Future research}

It appears to be possible to generalise the method used in this paper for number fields. That is, one can obtain pairs $(x_0,\Delta)$ such that there exists a prime ideal $\mathfrak{p}$ in a number field $\mathbb{K}$ with norm $N(\mathfrak{p})$ in the interval $(x(1-\Delta^{-1}),x]$; Hulse and Murty prove a similar result in \cite{HulseMurty}. In this generalisation, one would need to establish an analogous set-up to what we outlined in Section \ref{sec:backgound_theory}, then estimate a sum over the non-trivial zeros of the Dedekind zeta-function associated to $\mathbb{K}$, which is denoted $\zeta_{\mathbb{K}}$. There is no generalised Riemann height, so we could only use zero-free and zero-density regions of $\zeta_{\K}$ to obtain this estimate. The second author provides zero-free results in \cite{Lee}, and Hasanalizade \textit{et al.} provide the latest zero-density results in \cite{HasanalizadeShenWong}. Technically, a number fields generalisation of our Theorem \ref{thm:main} would be challenging to establish, but it would be an interesting addition to the literature. For example, authors such as Takeda \cite{Takeda} and Chattopadhyay \textit{et al.} \cite{ChattopadhyayEtAl} have used Hulse and Murty's work \cite{HulseMurty} to establish other results.


\bibliographystyle{amsplain} 
\bibliography{main}

\end{document}